\documentclass[letterpaper, 10 pt, conference]{ieeeconf}  

\IEEEoverridecommandlockouts                              

\usepackage{amsmath, amssymb}
\usepackage{amsfonts}
\usepackage{color}

\usepackage{cite}

\usepackage{graphicx}
\usepackage{epstopdf}
\usepackage{subfigure}

\usepackage{algorithmic}

\usepackage{xcolor}
\usepackage{tabularx}
\usepackage{multirow}

\newtheorem{theorem}{Theorem}
\newtheorem{lemma}{Lemma}

\newtheorem{assumption}{Assumption}
\newtheorem{definition}{Definition}

\newcommand{\R}{\mathbb{R}}

\newcommand{\barf}{f}

\newcommand{\mR}{\mathbf{R}}
\newcommand{\mC}{\mathbf{C}}
\newcommand{\mrR}{\mathrm{R}}
\newcommand{\mrC}{\mathrm{C}}

\newcommand{\bE}{\mathbb{E}}
\newcommand{\mI}{\mathbf{I}}

\newcommand{\T}{\intercal}

\newcommand{\inneighbor}[1]{\mathcal{N}^\mathrm{in}_{#1}}
\newcommand{\outneighbor}[1]{\mathcal{N}^\mathrm{out}_{#1}}

\newcommand{\mx}{\mathbf{x}}
\newcommand{\my}{\mathbf{y}}

\newcommand{\mM}{\mathbf{M}}
\newcommand{\ms}{\mathbf{s}}
\newcommand{\mepsilon}{\boldsymbol{\epsilon}}
\newcommand{\mxi}{\boldsymbol{\xi}}
\newcommand{\ox}{\bar{x}}
\newcommand{\oy}{\bar{y}}
\newcommand{\mW}{\mathbf{W}}

\title{\LARGE \bf
	A Robust Gradient Tracking Method for Distributed Optimization over Directed Networks
}

\author{Shi~Pu
	\thanks{This work was supported in parts by the Shenzhen Research Institute of Big Data Startup Fund No. J00120190011.}
	\thanks{S. Pu is with School of Data Science, Shenzhen Research Institute of Big Data, The Chinese University of Hong Kong, Shenzhen, China.
		{\tt\small (email: pushi@cuhk.edu.cn)}}%
}

\newcommand{\mA}{\mathbf{A}}
\newcommand{\mB}{\mathbf{B}}

\newcommand{\ta}{\tilde{\alpha}}

\def\sp#1{\textcolor{black}{#1}}

\begin{document}

	\maketitle
	\thispagestyle{empty}
	\pagestyle{empty}

	\begin{abstract}
		In this paper, we consider the problem of distributed consensus optimization over multi-agent networks with directed network topology. Assuming each agent has a local cost function that is smooth and strongly convex, the global objective is to minimize the average of all the local cost functions. To solve the problem, we introduce a robust gradient tracking method (R-Push-Pull) adapted from the recently proposed Push-Pull/AB algorithm \cite{pu2018push,xin2018linear}. R-Push-Pull inherits the advantages of Push-Pull and enjoys linear convergence to the optimal solution with exact communication.
		Under noisy information exchange, R-Push-Pull is more robust than the existing gradient tracking based algorithms; the solutions obtained by each agent reach a neighborhood of the optimum in expectation exponentially fast under a constant stepsize policy. We provide a numerical example that demonstrate the effectiveness of R-Push-Pull.
	\end{abstract}

	\section{Introduction}
	
	We consider a system of $n$ agents communicating through a network to collaboratively solve the following optimization problem:
	\begin{equation} \label{problem}
	\min\limits_{x\in\R^p}~\barf(x):=\frac{1}{n}\sum\limits_{i=1}^n f_i(x),
	\end{equation}
	where $x$ is the global decision variable and each function $f_i: \R^p\rightarrow \R$ is known by agent $i$ only. 
	The objective 
	is to obtain an optimal and consensual solution through local updates and local neighbor communications and information exchange. Such local exchange is desirable when sharing a large amount of data is prohibitively expensive due to limited communication resources, or when privacy needs to be preserved for individual agents. Scenarios in which problem (\ref{problem}) is considered include distributed machine learning \cite{nedic2017fast,cohen2017projected,forero2010consensus}, multi-agent target seeking \cite{pu2016noise,chen2012diffusion}, and wireless networks \cite{cohen2017distributed,mateos2012distributed,baingana2014proximal}, among many others.
	
	To solve problem~\eqref{problem} in a multi-agent system, 
	many distributed first-order algorithms have been proposed under various assumptions on the objective functions 
	and the underlying network topology~\cite{shi2015extra,seaman2017optimal,uribe2017optimal,xu2015augmented,di2016next,li2019decentralized,nedic2015distributed,xi2017add,zeng2017extrapush,xu2017convergence,pu2018swarming,pu2020distributed,nedic2017achieving,qu2017harnessing,xi2018linear}. Most works, including \cite{shi2015extra,seaman2017optimal,uribe2017optimal,xu2015augmented,di2016next,li2019decentralized,pu2020distributed}, 
	often restrict the network connectivity structure to undirected graphs, or more commonly require doubly stochastic mixing matrices.
	For strongly convex and smooth objective functions, EXTRA~\cite{shi2015extra} first uses a gradient difference structure to achieve the typical linear convergence rates of a centralized gradient method. 
	Recently, the gradient tracking technique has been employed to develop
	decentralized algorithms which are able to track the average of the gradients~\cite{xu2015augmented,di2016next,qu2017harnessing,nedic2017achieving,xi2017add,tian2020achieving} and enjoy linear convergence under possibly time-varying graphs.
	
	In this work, we are interested in the scenario where agents interact in a general directed network that includes undirected network as a special case.
	For directed graphs, constructing a doubly stochastic matrix needs weight balancing which requires an independent iterative process over the network. To avoid such a procedure, most efforts adopted the push-sum technique introduced in \cite{kempe2003gossip} for reaching average consensus over directed graphs. The work in \cite{tsianos2012push} first proposed a push-sum based distributed optimization algorithm for directed graphs. In~\cite{nedic2015distributed}, a push-sum based decentralized subgradient method was proposed and analyzed for time-varying directed graphs. For smooth objective functions, the work in~\cite{zeng2017extrapush,xi2017dextra} modifies the algorithm from~\cite{shi2015extra} with the push-sum technique, thus providing a new algorithm which converges linearly for smooth and strongly convex objective functions.
	
	Equipped with the gradient tracking technique, the algorithms developed in~\cite{xi2017add,xi2018linear,nedic2017achieving,xi2018linear,tian2020achieving} enjoy linear convergence for possibly time-varying directed graphs \cite{nedic2017achieving} or under asynchronous updates \cite{tian2020achieving}.
	Recently, the work in \cite{pu2018push,pu2020push,xin2018linear} introduced a modified gradient-tracking algorithm called Push-Pull/AB for distributed consensus optimization over directed networks. Unlike the push-sum protocol, Push-Pull uses a row stochastic matrix for the mixing of the decision variables, while it employs a column stochastic matrix for tracking the average gradients. It unifies different computational architectures and can be implemented asynchronously or over time-varying networks \cite{pu2020push,saadatniaki2020decentralized}. Accelerated version of the algorithm has also been developed \cite{xin2019distributed}.
	
	Despite that the aforementioned gradient tracking based algorithms are able to achieve linear convergence over directed networks under the smoothness and strong convexity condition, they are not robust to errors caused by noisy communication links or quantization \cite{srivastava2011distributed,cavalcante2013distributed,reisizadeh2018quantized,alistarh2017qsgd}. Considering noisy information exchange is extremely important in distributed optimization, for example, when one wishes to lower communication bandwidth costs among the agents by performing gradient compression techniques \cite{reisizadeh2018quantized,alistarh2017qsgd}, or when there is receiver-side noise corruption of signals in wireless networks \cite{cavalcante2013distributed}. As we will show both theoretically and through experiments, existing gradient tracking based methods fail in this scenario due to inaccurate tracking of the average gradients.
	
	To address the challenge, we propose a novel gradient tracking algorithm (R-Push-Pull) for distributed optimization that is robust to noisy information exchange. R-Push-Pull inherits the advantages of Push-Pull and achieves linear convergence to the optimal solution under noiseless communication links. It also allows for flexible network design and unifies different computational architectures.
	
	Our work is related to the literature in distributed stochastic optimization, where only stochastic gradient information is available (see e.g., \cite{nedic2016stochastic,pu2019sharp,pu2020asymptotic,koloskova2019decentralized,pu2020distributed,xin2019distributed2}) and distributed optimization using quantized information (see e.g., \cite{reisizadeh2018quantized,reisizadeh2019exact,taheri2020quantized}). For instance, the work in \cite{pu2020distributed,xin2019distributed} combined gradient tracking with stochastic gradient updates and show the iterates convergence linearly to a neighborhood of the optimal solution assuming strongly convex and smooth objectives. The paper \cite{reisizadeh2019exact} studied an exact quantized decentralized gradient descent algorithm which achieves a vanishing mean solution
	error under customary conditions for quantizers.
	A recent work \cite{taheri2020quantized} considered quantized push-sum for decentralized optimization over directed graphs and establishes subliear convergence to the optimal solution.

	
	\subsection{Main Contribution}
	
	Our main contribution of the paper is summarized as follows. Firstly, we introduce a novel gradient tracking method (R-Push-Pull) for distributed optimization over directed networks. The method achieves linear convergence to the optimal solution for minimizing the sum of smooth and strongly convex objective functions. Secondly, R-Push-Pull addresses the challenge of noisy information exchange.  It is shown to be more robust than the other linearly convergent, gradient tracking based algorithms such as Push-Pull/AB, in the sense that the solutions obtained by each agent running R-Push-Pull reach a neighborhood of the optimum in expectation exponentially fast under a constant stepsize policy, while the other competing algorithms lead to divergent solutions. Finally, we provide a numerical example that demonstrate the effectiveness of R-Push-Pull.
	
	\subsection{Notation}
	\label{subsec: notation}
	Vectors default to columns if not otherwise specified.
	Let each agent $i\in\{1,2,\ldots,n\}$ hold a local copy $x_i\in\mathbb{R}^p$ of the decision variable and an auxiliary variable $s_i\in\mathbb{R}^p$, where their
	values at iteration $k$ are denoted by $x_{i,k}$ and $s_{i,k}$, respectively. Denote
	\begin{eqnarray*}
		&\mx:=[x_1, x_2, \ldots,x_n]^{\T}\in\mathbb{R}^{n\times p},\\
		&\ms:=[s_1, s_2, \ldots,s_n]^{\T}\in\mathbb{R}^{n\times p}.
	\end{eqnarray*}
	Define $F(\mx)$ to be an aggregate objective function of the local variables,
	i.e., $F(\mx):=\sum_{i=1}^nf_i(x_i)$, and let
	\begin{equation*}
	\nabla F(\mx):=\left[\nabla f_1(x_1), \nabla f_2(x_2), \ldots, \nabla f_n(x_n)\right]^{\T}\in\mathbb{R}^{n\times p}.
	\end{equation*}
	
	\begin{definition}\label{def: norm n p}
		Given an arbitrary vector norm $\|\cdot\|$ on $\mathbb{R}^n$, for any $\mx\in\mathbb{R}^{n\times p}$, we define
		\begin{equation*}
		\|\mx\|:=\left\|\left[\|\mx^{(1)}\|,\|\mx^{(2)}\|,\ldots,\|\mx^{(p)}\|\right]\right\|_2,
		\end{equation*}
		where $\mx^{(1)},\mx^{(2)},\ldots,\mx^{(p)}\in\mathbb{R}^n$ are columns of $\mx$, and $\|\cdot\|_2$ represents the usual $2$-norm.
	\end{definition}
	
	\begin{definition}
		Given a square matrix $\mM$, its spectral radius is denoted by $\rho(\mM)$.
	\end{definition}
	
	
	\subsection{Organization}
		The rest of this paper is organized as follows. We state the problem of interest in Section \ref{sec: problem}. Then, we introduce the robust Push-Pull algorithm in Section~\ref{sec: method} along with the motivation. 
	We establish the convergence property of R-Push-Pull in Section~\ref{sec: analysis}.
	In Section~\ref{sec: numerical} we provide a simple numerical example. Section \ref{sec: conclusions} concludes the paper.

	\section{Problem Formulation}
	\label{sec: problem}
	
	We consider agents interact with each other in a general directed network (graph). 	A directed graph is a pair $\mathcal{G}=(\mathcal{V},\mathcal{E})$, where $\mathcal{V}$ is the set of vertices (nodes) and $\mathcal{E}\subseteq \mathcal{V}\times \mathcal{V}$ denotes the edge set consisted of ordered pairs of vertices.	
	Given a nonnegative matrix $\mathbf{M}=[m_{ij}]\in\mathbb{R}^{n\times n}$, the directed graph induced by $\mathbf{M}$ is 
	denoted by $\mathcal{G}_\mathbf{M}=(\mathcal{N},\mathcal{E}_\mathbf{M})$, where 
	$\mathcal{N}=\{1,2,\ldots,n\}$ and $(j,i)\in\mathcal{E}_\mathbf{M}$ if and only if $m_{ij}>0$.  For an arbitrary agent $i\in\mathcal{N}$, its in-neighbor set $\inneighbor{\mathbf{M},i}$ is defined as the collection of all individual agents that $i$ can actively and reliably pull data from in the graph $\mathcal{G}_\mathbf{M}$; we also define its out-neighbor set $\outneighbor{\mathbf{M},i}$ as the collection of all individual agents that can passively and reliably receive data from agent $i$.
	
	To solve Problem (\ref{problem}), assume each agent $i$ hold a local copy $x_i\in\mathbb{R}^p$ of the decision variable.
	Then Problem (\ref{problem}) can be written in the following equivalent form:
	\begin{equation} \label{problem_equiv}
	\begin{array}{c}
	\min\limits_{x_1,x_2\ldots,x_n\in\R^p}\sum\limits_{i=1}^n f_i(x_i)\\
	\text{s.t. }x_1=x_2=\cdots=x_n,
	\end{array}
	\end{equation}
	where the consensus constraint is imposed.
	
	Regarding the objective functions $f_i$ in problem~\eqref{problem}, we assume the following strong convexity and smoothness conditions.
	\begin{assumption}
		\label{asp; strconvex Lipschitz}
		Each $f_i$ is $\mu$-strongly convex with $L$-Lipschitz continuous gradients, i.e., for any $x,x'\in\mathbb{R}^p$,
		\begin{equation}
		\begin{split}
		& \langle \nabla f_i(x)-\nabla f_i(x'),x-x'\rangle\ge \mu\|x-x'\|^2,\\
		& \|\nabla f_i(x)-\nabla f_i(x')\|\le L \|x-x'\|.
		\end{split}
		\end{equation}
	\end{assumption}
	Under Assumption \ref{asp; strconvex Lipschitz}, problem~\eqref{problem} has a unique optimal 
	solution $x^*\in\mathbb{R}^{1\times p}$.

	
	\section{A Robust Gradient Tracking Method}
	\label{sec: method}
	
	We describe the robust Push-Pull Method in Algorithm 1, where $\epsilon_{i,k}\in\R^p$ and $\xi_{i,k}\in\R^p$ summarize the noise encountered by agent $i$ during the information exchange at step $k$. Sources of the noise include quantization \cite{alistarh2017qsgd} and/or noisy communication links \cite{reisizadeh2018quantized}.
		\smallskip	
	\begin{table}
		\normalsize
		\centering
		{\textbf{Algorithm 1: A Robust Push-Pull Method}}
		
		\smallskip
		\begin{tabularx}{0.5\textwidth}{X}
			\hline
			Choose stepsize $\alpha> 0$ and $\gamma,\eta\in(0,1]$,\\
			\quad in-bound mixing/pulling weights $R_{ij}\geq0$ for all $j\in\inneighbor{\mathbf{R},i}$,\\
			\quad and out-bound pushing weights $C_{li}\geq0$ for all $l\in\outneighbor{\mathbf{C},i}$;\\
			Each agent $i$ initializes with any arbitrary $x_{i,0},s_{i,0}\in\R^p$;\\
			\textbf{for} $k=0,1,\cdots$, \textbf{do} \\
			\quad for each $i\in\mathcal{N}$,\\
			\qquad agent $i$ pushes (noisy) $C_{li}s_{i,k}$ to each $l\in\outneighbor{\mathbf{C},i}$;\\
			\qquad agent $i$ pulls (noisy) $x_{j,k}$ from each $j\in\inneighbor{\mathbf{R},i}$;\\
			\quad for each $i\in\mathcal{N}$,\\
			\parbox{3cm}{\begin{align*}
				\begin{array}{ll}
					& s_{i,k+1} =  (1-\gamma)s_{i,k} + \gamma\left(\sum_{j=1}^n C_{ij}s_{j,k}+\epsilon_{i,k}\right)\\
				& \qquad\qquad+\nabla f_i(x_{i,k})\\
				& x_{i,k+1} =  (1-\eta)x_{i,k}+\eta\left(\sum_{j=1}^n R_{ij}x_{j,k}+\xi_{i,k}\right)\\
				&\qquad\qquad -\alpha (s_{i,k+1}-s_{i,k})
				\end{array}
				\end{align*}}\\
			\textbf{end for}\\
			\hline
		\end{tabularx}
	\end{table}
	Denote 
	\begin{equation*}
		\mepsilon_k:=[\epsilon_{1,k}, \epsilon_{2,k}, \ldots,\epsilon_{n,k}]^{\T},\quad
		\mxi_k:=[\xi_{1,k}, \xi_{2,k}, \ldots,\xi_{n,k}]^{\T}.
	\end{equation*}
	We make the following standing assumption.
	\begin{assumption}
		\label{asp: noises}
		\sp{Random sequences $\{\mepsilon_k\}$ and $\{\mxi_k\}$ are independent.\footnote{Note that at the same $k\ge 0$, $\epsilon_{i,k}$ (respectively, $\xi_{i,k}$) can be dependent among different agents.}} The matrices $\mepsilon_k$ and $\mxi_k$ have zero mean and bounded variance, i.e., $\bE[\mepsilon_k]=\bE[\mxi_k]=\mathbf{0}$, $\bE[\|\mepsilon_k\|^2]\le \sigma_{\mepsilon}^2$, $\bE[\|\mxi_k\|^2]\le \sigma_{\mxi}^2$ for some $\sigma_{\mepsilon},\sigma_{\mxi}>0$.
	\end{assumption}
	Assumption \ref{asp: noises} holds true, for example, when unbiased quantization is performed \cite{alistarh2017qsgd}.
	
	Denote
	\begin{equation}
	\begin{split}
	& \mC:=[C_{ij}],\, \mC_{\gamma}:=(1-\gamma)\mI+\gamma\mC,\\
	& \mR:=[R_{ij}],\, \mR_{\eta}:=(1-\eta)\mI+\eta\mR.
	\end{split}
	\end{equation}
	We can rewrite Algorithm 1 in the following matrix form:
	\begin{subequations}
		\label{R_pushpull}
		\begin{align}
		&\ms_{k+1} =  \mC_{\gamma}\ms_k+\gamma\mepsilon_k+\nabla F(\mx_k), \label{R_pushpull_s}\\
		&\mx_{k+1} =  \mR_{\eta}\mx_k+\eta\mxi_k-\alpha(\ms_{k+1}-\ms_k). \label{R_pushpull_x}
		\end{align}
	\end{subequations}
	
	The matrices $\mathbf{R}$ and $\mathbf{C}$ and their induced graphs $\mathcal{G}_\mathbf{R}$ and $\mathcal{G}_{\mathbf{C}}$ (respectively) satisfy the same conditions as for Push-Pull \cite{pu2020push}.
	\begin{assumption}
		\label{asp: stochastic}
		The matrix $\mathbf{R}\in\mathbb{R}^{n\times n}$ is nonnegative row-stochastic and $\mathbf{C}\in\mathbb{R}^{n\times n}$ is nonnegative column-stochastic, i.e., $\mathbf{R}\mathbf{1}=\mathbf{1}$ and $\mathbf{1}^{\T} \mathbf{C}=\mathbf{1}^{\T}$. In addition, the diagonal entries of $\mathbf{R}$ and $\mathbf{C}$ are positive. The graphs $\mathcal{G}_\mathbf{R}$ and $\mathcal{G}_{\mathbf{C}^{\T}}$ each contain at least one spanning tree. Moreover, there exists at least one node that is a root of spanning trees for both $\mathcal{G}_\mathbf{R}$ and $\mathcal{G}_{\mathbf{C}^{\T}}$.
	\end{assumption}

    The readers are referred to Section II of \cite{pu2020push} for the motivation of Assumption \ref{asp: stochastic}, which differs from most existing works on the assumption of network topology. As a result, R-Push-Pull is flexible in network design and unifies different computational architectures, including (semi-)centralized and decentralized architecture.

\begin{lemma}
	\label{lem: eigenvectors u v}
	Under Assumption~\ref{asp: stochastic}, the matrix $\mR$ has a nonnegative left eigenvector $u^{\T}$ (w.r.t.\ eigenvalue $1$) with $u^{\T}\mathbf{1}=n$, and the matrix $\mC$ has a nonnegative right eigenvector $v$ (w.r.t. eigenvalue $1$) with $\mathbf{1}^{\T}v=n$ (see \cite{horn1990matrix}).
\end{lemma}

   \subsection{Algorithm Development}
   
   To motivate the development of Algorithm 1, we first take a look at a variant of the Push-Pull algorithm in its matrix form (without noise):
   \begin{subequations}
   	\label{Push_Pull}
   	\begin{align}
   	&\mx_{k+1} =  \mathbf{R}\mx_k-\alpha\my_k,\\
   	& \my_{k+1} =  \mathbf{C}\my_k+\nabla F(\mx_{k+1})-\nabla F(\mx_k),
   	\end{align}
   \end{subequations}
   where $\mx_0$ is arbitrary and $\my_0=\nabla F(\mx_0)$. The matrices $\mR$ and $\mC$ satisfy Assumption \ref{asp: stochastic}. As a result of $\mathbf{C}$ being column-stochastic,  we have by induction that
   \begin{equation}
   \label{oy_k_tracking}
   \mathbf{1}^{\T}\my_k = \mathbf{1}^{\T}\nabla F(\mx_{k}), \qquad \forall k.
   \end{equation}
   Relation (\ref{oy_k_tracking}) is critical for (a subset of) the agents to track the average gradient $\mathbf{1}^{\T}\nabla F(\mx_{k})/n$ through the $\my$-update. However, this gradient tracking property is not robust. For example, if $y_{i,0}$ are not properly initialized such that $\mathbf{1}^{\T}\my_0\neq \mathbf{1}^{\T}\nabla F(\mx_0)$, then relation (\ref{oy_k_tracking}) will not hold for all $k>0$.\footnote{In fact, if $\mathbf{1}^{\T}\my_K\neq\mathbf{1}^{\T}\nabla F(\mx_{K})$  for any $K>0$, then relation (\ref{oy_k_tracking}) will not hold for all $k>K$.} Under the scenario of noisy information exchange, instead of relation (\ref{oy_k_tracking}), we have
   \begin{equation*}
   \mathbf{1}^{\T}\my_k = \mathbf{1}^{\T}\nabla F(\mx_{k})+\sum_{l=0}^{k-1}\mxi_l, \; \forall k,
   \end{equation*}
   which suggests that the gradient tracking will incur noise whose variance goes to infinity as $k$ grows.
   
   We remark here that although Push-Pull is not robust to noisy information exchange, it works well with stochastic gradient information under noiseless communication \cite{xin2019distributed2,pu2020distributed}. This is due to the fact that
   \begin{equation}
   \label{oy_k_stochastic}
   \mathbf{1}^{\T}\my_k = \mathbf{1}^{\T}G(\mx_{k}), \; \forall k,
   \end{equation}
   where $G(\mx_{k})$ stands for an unbiased estimate of $\nabla F(\mx_k)$. As a result, the (stochastic) gradient tracking is still effective given that $G(\mx_{k})$ has bounded variance. We will see below that after a variable transformation, R-Push-Pull resembles Push-Pull while ensuring gradient tracking in the form of (\ref{oy_k_stochastic}).
   
   Let $\my_k:=\ms_{k+1}-\ms_k,\; \forall k$. We have
   \begin{subequations}
   	\label{R_pushpull_transform}
   	\begin{align}
   	&\mx_{k+1} =  \mR_{\eta} \mx_k+\eta\mxi_k-\alpha\my_k, 	\label{R_pushpull_transform_x}\\
   	& \my_{k+1} =  \mathbf{C}_{\gamma}\my_k+\tilde{\nabla} F(\mx_{k+1})-\tilde{\nabla} F(\mx_k), \label{R_pushpull_transform_y}
   	\end{align}
   \end{subequations}
   where 
\begin{equation*}
\tilde{\nabla}F(\mx_k):= \nabla F(\mx_k)+\gamma\mepsilon_k,\; \forall k.
\end{equation*}
   This is exactly the Push-Pull update (\ref{Push_Pull}) if we ignore the noise and let $\gamma=\eta=1$.
   
   To see why R-Push-Pull is robust in gradient tracking, note that from (\ref{R_pushpull_s}), we have by induction that 
   \begin{equation*}
   \mathbf{1}^{\T}\ms_k = \mathbf{1}^{\T}\ms_0 + \sum_{l=0}^{k-1}\mathbf{1}^{\T}\tilde{\nabla} F(\mx_{k}), \; \forall k.
   \end{equation*}
   Hence $\ms_k$ tracks the aggregated gradients of the network over the history. Then from the definition of $\my_k$, we have 
   	\begin{equation*}
   	\my_0 = \ms_1-\ms_0=\mC_{\gamma} \ms_0+\tilde{\nabla}F(\mx_0)-\ms_0,
   	\end{equation*}
   	indicating that $\mathbf{1}^{\T}\my_0=\mathbf{1}^{\T}\tilde{\nabla}F(\mx_0)$
   	regardless of the initial choice $\ms_0$. In addition,
   	\begin{equation}
   	\label{y_k_sum}
   	\mathbf{1}^{\T}\my_k = \mathbf{1}^{\T}\ms_{k+1}-\mathbf{1}^{\T}\ms_k = \mathbf{1}^{\T}\tilde{\nabla} F(\mx_{k}), \; \forall k,
   	\end{equation}
   	whose variance is bounded just like in equation (\ref{oy_k_stochastic}).
   	
   	It is worth noting that the robust gradient tracking technique employed by R-Push-Pull, i.e., using $\ms_k$ and $\ms_{k+1}-\ms_k$ to track the aggregated history gradients and the average gradient at step $k$ respectively, can also be applied to other gradient tracking based methods such as Push-DIGing/ADDOPT \cite{nedic2017achieving,xi2018add}.
   	
	
	
	\section{Convergence Analysis}
	\label{sec: analysis}
	
	In this section, we study the convergence properties of R-Push-Pull.
	We first define the following variables:
	\begin{eqnarray*}
		\ox_k  :=  \frac{1}{n}u^{\T} \mx_k,\ \
		\oy_k  :=  \frac{1}{n}\mathbf{1}^{\T}\my_k.
	\end{eqnarray*}
	Our strategy is to bound $\bE[\|\ox_{k+1}-x^*\|_2^2]$, $\bE[\|\mx_{k+1}-\mathbf{1}\ox_{k+1}\|_{\mrR}^2]$ and $\bE[\|\my_{k+1}-v\oy_{k+1}\|_{\mrC}^2]$ in terms of linear combinations of their previous values, where $\|\cdot\|_{\mrR}$ and $\|\cdot\|_{\mrC}$ are specific norms to be defined later. In this way we establish a linear system of inequalities which allows us to derive the convergence results. The proof technique is similar to that of \cite{xin2019distributed,pu2020push} and was inspired by earlier works \cite{qu2017harnessing,xi2018linear}.
	
	\subsection{Preliminary Analysis}
	
	From relation (\ref{R_pushpull_transform_x}) and Lemma \ref{lem: eigenvectors u v}, we have
	\begin{equation}
	\label{ox_k+1 pre}
	\ox_{k+1} = \frac{1}{n}u^{\T} (\mR_{\eta}\mx_k+\eta\mxi_k-\alpha \my_k)=\ox_k-\frac{\alpha}{n}u^{\T}\my_k+\frac{\eta}{n}u^{\T}\mxi_k.
	\end{equation}
	By relation \eqref{y_k_sum},
	\begin{equation}
	\label{oy_ave}
	\oy_k = \frac{1}{n}\mathbf{1}^{\T}\tilde{\nabla} F(\mx_{k}), \ \ \forall k.
	\end{equation}
	Let us further define 
	\begin{align*}
	h_k := \frac{1}{n}\mathbf{1}^{\T}\nabla F(\mx_k)=\frac{1}{n}\sum_{i=1}^n f_i(x_{i,k}),\\
	g_k := \frac{1}{n}\mathbf{1}^{\T}\nabla F(\mathbf{1}\ox_k) = \frac{1}{n}\sum_{i=1}^n f_i(\bar{x}_k).
	\end{align*}
	Clearly, $\oy_k=h_k+\frac{\gamma}{n}\mathbf{1}^{\T}\mepsilon_k$.
	Then, we obtain from relation (\ref{ox_k+1 pre}) that
	\begin{equation}
		\label{ox pre}
		\begin{array}{l}
		\ox_{k+1} = \ox_k-\frac{\alpha}{n}u^{\T}\left(\my_k-v\oy_k+v\oy_k\right)+\frac{\eta}{n}u^{\T}\mxi_k\\
		= \ox_k-\frac{\alpha}{n}u^{\T}v\oy_k-\frac{\alpha}{n}u^{\T}\left(\my_k-v\oy_k\right)+\frac{\eta}{n}u^{\T}\mxi_k\\
		= \ox_k-\ta g_k-\ta (h_k-g_k)-\frac{\alpha}{n}u^{\T}\left(\my_k-v\oy_k\right)\\
		\quad -\ta\frac{\gamma}{n}\mathbf{1}^{\T}\mepsilon_k+\frac{\eta}{n}u^{\T}\mxi_k,
		\end{array}
	\end{equation}
	where 
	\begin{equation}
	\label{alpha'}
	\ta :=\frac{\alpha}{n} u^{\T}v.\footnote{Assumption \ref{asp: stochastic} ensures $\ta >0$ (see \cite{pu2018distributed}).}
	\end{equation}
	
	In view of \eqref{R_pushpull_transform_x} and Lemma \ref{lem: eigenvectors u v}, using \eqref{ox_k+1 pre} we have
	\begin{equation}
	\label{mx-ox pre}
	\begin{array}{ll}
	\mx_{k+1}-\mathbf{1}\ox_{k+1} \\
	= \mR_{\eta}\mx_k+\eta\mxi_k-\alpha \my_k-\mathbf{1}\ox_k+\alpha\frac{\mathbf{1}u^{\T}}{n}\my_k-\frac{\mathbf{\eta}u^{\T}}{n}\mxi_k\\
	=  \mR_{\eta}(\mx_k-\mathbf{1}\ox_k)-\alpha\left(\mI-\frac{\mathbf{1}u^{\T}}{n}\right)\my_k+\eta(\mI-\frac{\mathbf{1}u^{\T}}{n})\mxi_k\\
	=  \left(\mR_{\eta}-\frac{\mathbf{1}u^{\T}}{n}\right)(\mx_k-\mathbf{1}\ox_k)-\alpha\left(\mI-\frac{\mathbf{1}u^{\T}}{n}\right)\my_k\\
	\quad +\eta(\mI-\frac{\mathbf{1}u^{\T}}{n})\mxi_k,
	\end{array}
	\end{equation}
	and from \eqref{R_pushpull_transform_y} and \eqref{oy_ave} we obtain
	\begin{equation}
	\label{my-oy pre}
	\begin{array}{ll}
	\my_{k+1}-v\oy_{k+1} \\
	= \mC_{\gamma}\my_k+\tilde{\nabla} F(\mx_{k+1})-\tilde{\nabla} F(\mx_k)-v(\oy_{k+1}-\oy_k)-v\oy_k\\
	= \mC_{\gamma}\my_k-v\oy_k+\left(\mI-\frac{v\mathbf{1}^{\T}}{n}\right)\left(\tilde{\nabla} F(\mx_{k+1})-\tilde{\nabla} F(\mx_k)\right)\\
	=\left(\mC_{\gamma}-\frac{v\mathbf{1}^{\T}}{n}\right)(\my_k-v\oy_k)\\
	\quad +\left(\mI-\frac{v\mathbf{1}^{\T}}{n}\right)\left(\tilde{\nabla} F(\mx_{k+1})-\tilde{\nabla} F(\mx_k)\right).
	\end{array}
	\end{equation}
	
	Denote by $\mathcal{F}_k$ the $\sigma$-algebra generated by $\{\mepsilon_0,\mxi_0\ldots,\mepsilon_{k-1},\mxi_{k-1}\}$, and define $\bE[\cdot \mid\mathcal{F}_k]$ as the conditional expectation given $\mathcal{F}_k$. We prepare a few useful supporting lemmas for our further analysis, deferred to Appendix \ref{subsec: supporting lemmas}.

	\subsection{Main Results}
	The following lemma establishes a linear system of inequalities that bound $\bE[\|\ox_{k+1}-x^*\|_2^2]$, $\bE[\|\mx_{k+1}-\mathbf{1}\ox_k\|_{\mrR}^2]$ and $\bE[\|\my_{k+1}-v\oy_k\|_{\mrC}^2]$.
	\begin{lemma}
		\label{lem: important inequalities}
		Under Assumptions \ref{asp; strconvex Lipschitz}-\ref{asp: stochastic}, when $\ta \le 1/(\mu+L)$, we have the following linear system of inequalities:
		\begin{equation}
		\label{main ineqalities}
		\begin{bmatrix}
		\bE[\|\ox_{k+1}-x^*\|_2^2]\\
		\bE[\|\mx_{k+1}-\mathbf{1}\ox_{k+1}\|_{\mrR}^2]\\
		\bE[\|\my_{k+1}-v\oy_{k+1}\|_{\mrC}^2]
		\end{bmatrix}
		\le
		\mA
		\begin{bmatrix}
		\bE[\|\ox_k-x^*\|_2^2]\\
		\bE[\|\mx_k-\mathbf{1}\ox_k\|_{\mrR}^2]\\
		\bE[\|\my_k-v\oy_k\|_{\mrC}^2]
		\end{bmatrix}+\mB,
		\end{equation}
		where the inequality is to be taken component-wise. The transition matrix $\mA=[a_{ij}]$ and the vector $\mB$ are given by:
		\begin{equation}
		\mA = \begin{bmatrix}
		1-\ta \mu & c_1\alpha & c_2\alpha\\
		c_5\alpha^2 & \frac{1+\tau_{\mrR}^2}{2}+c_6\alpha^2 & c_7\alpha^2\\
		c_{10}\alpha^2 & c_{11} & \frac{1+\tau_{\mrC}^2}{2}+c_{12}\alpha^2
		\end{bmatrix}
		\end{equation}
		and
		\begin{equation}
		\mB=\begin{bmatrix}
		c_3\alpha^2\\
		c_8\alpha^2\\
		c_{13}
		\end{bmatrix}
		\gamma^2\sigma_{\mepsilon}^2
		+
		\begin{bmatrix}
		c_4\\
		c_9\\
		c_{14}
		\end{bmatrix}
		\eta^2\sigma_{\mxi}^2,
		\end{equation}
		respectively, where constants $c_1-c_{13}$ are defined in (\ref{c1-c4}), \eqref{c5-c9}, and \eqref{c10-c13}.
	\end{lemma}
	\begin{proof}
		See Appendix \ref{proof lem: important inequalities}.
	\end{proof}
		The following theorem established the convergence properties for R-Push-Pull in \eqref{R_pushpull}.
	\begin{theorem}
		\label{Theorem1}
		Suppose Assumptions \ref{asp; strconvex Lipschitz}-\ref{asp: stochastic} hold and 
		the stepsize $\alpha$ satisfies
		\begin{equation}
		\label{alpha_ultimate_bound}
		\alpha\le \min\left\{\sqrt{\frac{1-\tau_{\mrR}^2}{6c_6}},\sqrt{\frac{1-\tau_{\mrC}^2}{6c_{12}}},\sqrt{\frac{2d_3}{d_2+\sqrt{d_2^2+4d_1d_3}}}\right\},
		\end{equation}
		where $d_1-d_3$ are defined in \eqref{d1-d3}.
		Then $\sup_{l\ge k}\bE[\|\ox_l-x^*\|_2^2]$ and $\sup_{l\ge k}\bE[\|\mx_{l}-\mathbf{1}\ox_{l}\|_{\mrR}^2]$, respectively, converge to $\limsup_{k\rightarrow\infty}\bE[\|\ox_k-x^*\|_2^2]$ and $\limsup_{k\rightarrow\infty}\bE[\|\mx_k-\mathbf{1}\ox_k\|_{\mrR}^2]$ at the linear rate $\mathcal{O}(\rho(\mA)^k)$, where $\rho(\mA)<1$ is the spectral radius of the matrix $\mA$. Furthermore,
		\begin{equation}
		\label{limsup_pre}
		\begin{split}
		& \limsup_{k\rightarrow\infty}\bE[\|\ox_k-x^*\|_2^2] \le [(\mI-\mA)^{-1}\mB]_1,\\
		& \limsup_{k\rightarrow\infty}\bE[\|\mx_k-\mathbf{1}\ox_k\|_{\mrR}^2] \le [(\mI-\mA)^{-1}\mB]_2,
		\end{split}
		\end{equation}
		where $[(\mI-\mA)^{-1}\mB]_j$ denotes the $j$-th element of the vector $[(\mI-\mA)^{-1}\mB]$. Their specific forms are given in (\ref{opt_error_bound}) and (\ref{consensus_error_bound}) respectively.
	\end{theorem}
	\begin{proof}
		In light of Lemma \ref{lem: important inequalities}, by induction we have
		{\small\begin{eqnarray}
			\label{linear_system_bound}
			\begin{bmatrix}
			\bE[\|\ox_{k}-x^*\|_2^2]\\
			\bE[\|\mx_{k}-\mathbf{1}\ox_{k}\|_{\mrR}^2]\\
			\bE[\|\my_{k}-v\oy_{k}\|_{\mrC}^2]
			\end{bmatrix}
			\le 
			\mA^k\begin{bmatrix}
			\bE[\|\ox_{0}-x^*\|_2^2]\\
			\bE[\|\mx_{0}-\mathbf{1}\ox_{0}\|_{\mrR}^2]\\
			\bE[\|\my_{0}-v\oy_{0}\|_{\mrC}^2]
			\end{bmatrix}+\sum_{l=0}^{k-1}\mA^l\mB.
			\end{eqnarray}}\normalsize
		If the spectral radius of $\mA$ satisfies $\rho(\mA)<1$, then $\mA^k$ converges to $\mathbf{0}$ at the linear rate $\mathcal{O}(\rho(\mA)^k)$  (see \cite{horn1990matrix}), in which case $\sup_{l\ge k}\bE[\|\ox_l-x^*\|^2]$, $\sup_{l\ge k}\bE[\|\mx_l-\mathbf{1}\ox_l\|^2]$ and $\sup_{l\ge k}\bE[\|\my_l-\mathbf{1}\oy_l\|^2]$ all converge to a neighborhood of $0$ at the linear rate $\mathcal{O}(\rho(\mA)^k)$. 
		
		The next lemma provides conditions that ensure $\rho(\mA)<1$.
		\begin{lemma}
			\label{lem: rho_M}
			(Lemma 5 in \cite{pu2020distributed}) Given a nonnegative, irreducible matrix $\mM=[m_{ij}]\in\mathbb{R}^{3\times 3}$ with $m_{11},m_{22},m_{33}<\lambda^*$ for some $\lambda^*>0$. A necessary and sufficient condition for $\rho(\mM)<\lambda^*$ is $\text{det}(\lambda^* \mI-\mM)>0$.
		\end{lemma}
	
		In light of Lemma \ref{lem: rho_M}, it suffices to ensure $a_{11},a_{22},a_{33}<1$ and $\text{det}(\mI-\mA)>0$, or more aggressively,
		{\begin{equation}
			\label{|I-A|>0 pre}
			\begin{split}
			& \text{det}(\mI-\mA)=(1-a_{11})(1-a_{22})(1-a_{33})-a_{12}a_{23}a_{31}\\
			& \quad -a_{13}a_{21}a_{32}-(1-a_{22})a_{13}a_{31}-(1-a_{11})a_{23}a_{32}\\
			& \quad -(1-a_{33})a_{12}a_{21}>\frac{1}{2}(1-a_{11})(1-a_{22})(1-a_{33}),
			\end{split}
			\end{equation}}\normalsize
		which is equivalent to
			\begin{equation}
			\label{|I-A|>0}
			\begin{split}
			& \frac{1}{2}(1-a_{11})(1-a_{22})(1-a_{33})-c_1 c_7 c_{10}\alpha^5\\
			& \quad -c_2 c_5 c_{11}\alpha^3-(1-a_{22})c_2c_{10}\alpha^3-(1-a_{11})c_7 c_{11}\alpha^2\\
			& \quad -(1-a_{33})c_1 c_5 \alpha^3>0.
			\end{split}
			\end{equation}\normalsize
		We now provide some sufficient conditions under which $a_{11},a_{22},a_{33}<1$ and (\ref{|I-A|>0}) holds true.
		
		First, $a_{11}<1$ is ensured by choosing $\ta \le 1/(\mu+L)$, and $a_{22},a_{33}<1$ is guaranteed by
		\begin{equation}
		\label{1-a22,a33ge}
		1-a_{22} \ge \frac{1}{3}(1-\tau_{\mrR}^2),\, 1-a_{33} \ge \frac{1}{3}(1-\tau_{\mrC}^2),
		\end{equation}
		which requires
		\begin{equation}
		\label{alpha loose condition}
		\alpha\le \min\left\{\sqrt{\frac{1-\tau_{\mrR}^2}{6c_6}},\sqrt{\frac{1-\tau_{\mrC}^2}{6c_{12}}}\right\}.
		\end{equation}
		
		Second, notice that $a_{22}>\frac{1+\tau_{\mrR}^2}{2}$ and $a_{33}>\frac{1+\tau_{\mrC}^2}{2}$. In light of \eqref{1-a22,a33ge}, a sufficient condition for $\text{det}(\mI-\mA)>0$ is to substitute the first $(1-a_{22})$ (respectively, $(1-a_{33})$) in (\ref{|I-A|>0}) by $(1-\tau_{\mrR}^2)/3$ (respectively, $(1-\tau_{\mrC}^2)/3$), and substitute the second $(1-a_{22})$ (respectively, $(1-a_{33})$) by $(1-\tau_{\mrR}^2)/2$ (respectively, $(1-\tau_{\mrC}^2)/2$). We then have 
		\begin{equation*}
		d_1\alpha^4+d_2\alpha^2-d_3<0,
		\end{equation*}
		where
		\begin{equation}
		\label{d1-d3}
		\begin{split}
		& d_1 := c_1 c_7 c_{10}\\
		& d_2 := c_2 c_5 c_{11}+\frac{(1-\tau_{\mrR}^2)}{2}c_2 c_{10}+\mu u^{\T}v c_7 c_{11}\\
		& \quad +\frac{(1-\tau_{\mrC}^2)}{2}c_1 c_5\\
		& d_3 := \frac{\mu u^{\T}v(1-\tau_{\mrR}^2)(1-\tau_{\mrC}^2)}{18}.
		\end{split}
		\end{equation}
		Hence
		\begin{equation}
		\label{alpha strict condition}
		\alpha^2\le \frac{2d_3}{d_2+\sqrt{d_2^2+4d_1d_3}}.
		\end{equation}
		Relations (\ref{alpha loose condition}) and (\ref{alpha strict condition}) yield the final bound on $\alpha$.
		
		Relation \eqref{limsup_pre} follows from (\ref{linear_system_bound}) directly given that $\rho(\mA)<1$. In light of \eqref{|I-A|>0 pre} and \eqref{1-a22,a33ge}, we obtain from (\ref{linear_system_bound}) that
		{\small\begin{equation}
			\label{opt_error_bound}
			\begin{split}
			&  [(\mI-\mA)^{-1}\mB]_1\\
			& = \left[((1-a_{22})(1-a_{33})-a_{23}a_{32})(c_3\alpha^2\gamma^2\sigma_{\mepsilon}^2+c_4\eta^2\sigma_{\mxi}^2)\right.\\
			& \quad +(a_{13}a_{32}+a_{12}(1-a_{33}))(c_8\alpha^2\gamma^2\sigma_{\mepsilon}^2+c_9\eta^2\sigma_{\mxi}^2))\\
			& \quad \left. +(a_{12}a_{23}+a_{13}(1-a_{22}))(c_{13}\gamma^2\sigma_{\mepsilon}^2+c_{14}\eta^2\sigma_{\mxi}^2))\right]\frac{1}{\text{det}(\mI-\mA)}\\
			& \le \left[\left(\frac{(1-\tau_{\mrR}^2)(1-\tau_{\mrC}^2)}{4}-c_7 c_{11}\alpha^2\right)(c_3\alpha^2\gamma^2\sigma_{\mepsilon}^2+c_4\eta^2\sigma_{\mxi}^2))\right.\\
			& \quad +\left(c_2c_{11}\alpha+\frac{c_1(1-\tau_{\mrC}^2)\alpha}{2}\right)(c_8\alpha^2\gamma^2\sigma_{\mepsilon}^2+c_9\eta^2\sigma_{\mxi}^2))\\
			& \quad \left. +\left(c_1c_7\alpha^3+\frac{c_2(1-\tau_{\mrR}^2)\alpha}{2}\right)(c_{13}\gamma^2\sigma_{\mepsilon}^2+c_{14}\eta^2\sigma_{\mxi}^2))\right]\\
			& \quad \cdot\frac{18}{\ta \mu(1-\tau_{\mrR}^2)(1-\tau_{\mrC}^2)},
			\end{split}
			\end{equation}}\normalsize
		and
		{\small\begin{equation}
			\label{consensus_error_bound}
			\begin{split}
			&  [(\mI-\mA)^{-1}\mB]_2\\
			& =\left[(a_{23}a_{31}+a_{21}(1-a_{33}))(c_3\alpha^2\gamma^2\sigma_{\mepsilon}^2+c_4\eta^2\sigma_{\mxi}^2)\right.\\
			& \quad  +((1-a_{11})(1-a_{33})-a_{13}a_{31})(c_8\alpha^2\gamma^2\sigma_{\mepsilon}^2+c_9\eta^2\sigma_{\mxi}^2)\\
			& \quad \left. +(a_{13}a_{21}+a_{23}(1-a_{11}))(c_{13}\gamma^2\sigma_{\mepsilon}^2+c_{14}\eta^2\sigma_{\mxi}^2)\right]\frac{1}{\text{det}(\mI-\mA)}\\
			& \le \left[\left(c_7c_{10}\alpha^4+c_5\alpha^2\frac{(1-\tau_{\mrC}^2)}{2}\right)(c_3\alpha^2\gamma^2\sigma_{\mepsilon}^2+c_4\eta^2\sigma_{\mxi}^2)\right.\\
			& \quad +\left(\ta \mu\frac{(1-\tau_{\mrC}^2)}{2}-c_2c_{10}\alpha^3\right)(c_8\alpha^2\gamma^2\sigma_{\mepsilon}^2+c_9\eta^2\sigma_{\mxi}^2)\\
			& \quad \left. +\left(c_2 c_5\alpha^3+c_7\alpha^2\ta \mu\right)(c_{13}\gamma^2\sigma_{\mepsilon}^2+c_{14}\eta^2\sigma_{\mxi}^2)\right]\\
			& \quad \cdot\frac{18}{\ta \mu(1-\tau_{\mrR}^2)(1-\tau_{\mrC}^2)}.\\
			\end{split}
			\end{equation}}\normalsize
	\end{proof}
	From Theorem \ref{Theorem1}, if $\sigma_{\mepsilon}=\sigma_{\mxi}=\mathbf{0}$ (no noise), we have $\mB=\mathbf{0}$, then R-Push-Pull converges linearly to the optimal solution $x^*$. Specifically, when $\alpha$ is sufficiently small, it can be shown that the linear rate indicator $\rho(\mA)\simeq 1-\ta \mu$.
	
	The upper bounds in (\ref{opt_error_bound}) and (\ref{consensus_error_bound}) are functions of $\alpha$, $\eta$, $\gamma$ and other problem parameters, and they are decreasing in the variances $\sigma_{\mepsilon}$ and $\sigma_{\mxi}$. It can also be numerically verified that $\limsup_{k\rightarrow\infty}\bE[\|\ox_k-x^*\|_2^2]$
	and $\limsup_{k\rightarrow\infty}\bE[\|\mx_k-\mathbf{1}\ox_k\|_{\mrR}^2]$ decrease in the stepsize $\alpha$.
	Moreover, in order to mitigate the effect of noise $\mepsilon_k$ on the final optimization error, we may take $\eta$ to be in the order of $\mathcal{O}(\alpha)$.

	
	\section{Numerical Example}
	\label{sec: numerical}
	
	We provide a simple illustration example. Consider the Ridge regression problem, i.e.,
	\begin{equation}
	\label{Ridge Regression}
	\min_{x\in \mathbb{R}^{p}}f(x)=\frac{1}{n}\sum_{i=1}^nf_i(x)\left(=\left(u_i^{\T} x-v_i\right)^2+\rho\|x\|^2\right),
	\end{equation}
	where $\rho>0$ is a penalty parameter.
	Each agent $i$ has sample $(u_i,v_i)$ with $u_i\in\mathbb{R}^p$ representing the features and $v_i\in\mathbb{R}$ being the observed outputs. Let each $u_i\in[-1,1]^p$ be generated from uniform distribution, and $v_i$ is drawn according to $v_i=u_i^{\T} \tilde{x}_i+\varepsilon_i$, where parameters $\tilde{x}_i$ are  evenly located in $[0,10]^p$, and $\varepsilon_i\sim\mathcal{N}(0,25)$.\footnote{These randomly generated problem parameters have negligible effects on the simulation results.} Given these predetermined parameters,
	Problem (\ref{Ridge Regression}) has a unique solution $x^*=(\sum_{i=1}^n[u_iu_i^{\T}]+n\rho\mathbf{I})^{-1}\sum_{i=1}^n[u_iu_i^{\T}]\tilde{x}_i$.
	
	We compare the performance of R-Push-Pull against other two gradient tracking based algorithms Push-Pull/AB \cite{pu2018push,xin2018linear} and Push-DIGing/ADDOPT \cite{nedic2017achieving,xi2018add}. 
	To model noisy information exchange, we assume the transmitted values such as $x_{i,k}$, $C_{li}y_{i,k}$ and $C_{li}s_{i,k}$ are corrupted with independent Gaussian noises $\mathcal{N}(0,0.01)$.
	Regarding the network topology, we generate a directed graph $\mathcal{G}$ of $15$ nodes by adding random links to a ring network, where a directed link exists between any two nonadjacent nodes with probability $0.3$. Then for R-Push-Pull and Push-Pull, we let $\mathcal{G}_{\mR}=\mathcal{G}_{\mC}=\mathcal{G}$ for simplicity. 
	
	Design matrix $\mC$ as follows: for any agent $i$,
	$C_{li}=\frac{1}{|\outneighbor{\mathbf{C},i}|+1}$ for all $l\in\outneighbor{\mathbf{C},i}$ and $C_{ii}=1-\sum_{l\in \outneighbor{\mathbf{C},i}}C_{li}$. Letting $\gamma=0.5$, the same $\mC_{\gamma}$ is used for all the algorithms.
	In R-Push-Pull and Push-Pull, for any agent $i$, $R_{ij}=\frac{1}{|\inneighbor{\mathbf{R},i}|+1}$ for all $j\in \inneighbor{\mathbf{R},i}$, and $R_{ii}=1-\sum_{j\in \inneighbor{\mathbf{R},i}}R_{ij}$. The matrix $\mR_{\eta}$ is constructed by taking $\eta=0.01$.
	
	\begin{figure}[htbp]
	\centering
	\includegraphics[width=3.2in]{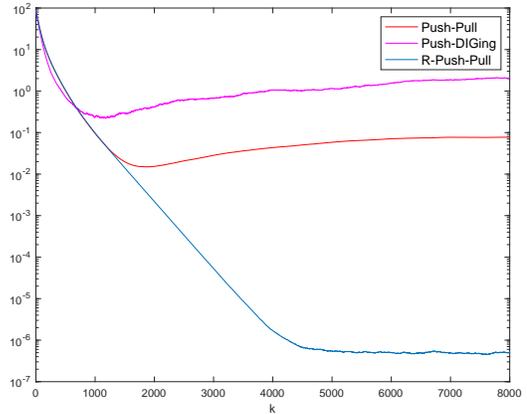} 
		\caption{Comparison of the performance of R-Push-Pull, Push-Pull/AB and Push-DIGing/ADDOPT, measured by $\frac{1}{n}\bE[\|x_{i,k}-x^*\|^2]$. The expected errors are approximated by averaging over $50$ simulation results. Dimension $p=10$, stepsize $\alpha=0.01$ and penalty parameter $\rho=0.01$.}
		\label{Figure_comparison}
	\end{figure}

Fig. \ref{Figure_comparison} compares the performance of different algorithms with respect to $\frac{1}{n}\bE[\|x_{i,k}-x^*\|^2]$.
It can be seen that initially all the errors decrease exponentially fast at comparable rates. However, the errors for Push-Pull/AB and Push-DIGing/ADDOPT eventually increase over time while R-Push-Pull achieves linear convergence to a small neighborhood of $0$. This sharp contrast verifies the effectiveness of the proposed algorithm.
	
	
	\section{Conclusions}
	\label{sec: conclusions}
	In this paper, we introduce a robust gradient tracking method (R-Push-Pull) for distributed optimization over directed networks.
	R-Push-Pull inherits the advantage of Push-pull and achieves linear convergence to the optimal solution with exact information fusion.
	Under noisy information exchange, R-Push-Pull is more robust than the other gradient tracking based algorithms. We show the solutions obtained by each agent reach a neighborhood of the optimum in expectation exponentially fast under a constant stepsize policy. We also provide a numerical example that demonstrate the effectiveness of R-Push-Pull.
	
	
	\section*{ACKNOWLEDGMENT}
	We would like to thank Wei Shi from Princeton University and Jinming Xu from Zhejiang University for helpful discussions.
	
	
	\bibliographystyle{IEEEtran}
	\bibliography{mybib}
	
	\section{APPENDIX}
	
	\subsection{Supporting Lemmas}
	\label{subsec: supporting lemmas}
	\begin{lemma}
		\label{lem: Lipschitz implications}
		Under Assumption \ref{asp; strconvex Lipschitz}, there holds
		\begin{subequations}
			\begin{align}
			& \|h_k-g_k\|_2 \le \frac{L}{\sqrt{n}}\|\mx_k-\mathbf{1}\ox_k\|_2, \label{h_k-g_k}\\
			& \bE[\|\oy_k-h_k\|_2^2\mid \mathcal{F}_k] \le \frac{\gamma^2\sigma_{\mepsilon}^2}{n}, \label{oy_k-h_k}\\
			& \|g_k\|_2 \le L\|x_k-x^*\|_2. \label{g_k}
			\end{align}
		\end{subequations}
		In addition, when $\ta \le 2/(\mu+L)$, we have
		\begin{equation}
		\label{opt_contraction}
		\|\ox_k-\ta g_k-x^*\|_2 \le (1-\ta \mu)\|\ox_k-x^*\|_2, \ \ \forall k.
		\end{equation}
	\end{lemma}
	\begin{proof}
		In light of Assumption \ref{asp; strconvex Lipschitz},
		\begin{multline*}
		\|h_k-g_k\|_2=\frac{1}{n}\|\mathbf{1}^{\T}\nabla F(\mx_{k})-\mathbf{1}^{\T}\nabla F(\mathbf{1}\ox_k)\|_2\\
		\le \frac{L}{n}\sum_{i=1}^n \|x_{i,k}-\ox_k\|_2\le \frac{L}{\sqrt{n}}\|\mx_k-\mathbf{1}\ox_k\|_2.
		\end{multline*}
		From (\ref{oy_ave}) and Assumption \ref{asp: noises}, we have
		\sp{\begin{equation*}
			\bE[\|\bar{y}_k-h_k\|_2^2\mid \mathcal{F}_k] = \frac{\gamma^2}{n^2}\bE[\|\mathbf{1}^{\T}\mepsilon_k\|_2^2\mid \mathcal{F}_k] \le \frac{\gamma^2\sigma_{\mepsilon}^2}{n}.
			\end{equation*}}
		Finally, Assumption \ref{asp; strconvex Lipschitz} leads to
		\begin{multline*}
		\|g_k\|_2 = \frac{1}{n}\|\mathbf{1}^{\T}\nabla F(\mathbf{1}\ox_k)-\mathbf{1}^{\T}\nabla F(\mathbf{1}x^*)\|_2\\
		\le \frac{L}{n}\sum_{i=1}^n \|\ox_k-x^*\|_2 = L\|\ox_k-x^*\|_2.
		\end{multline*}
		Proof of the relation (\ref{opt_contraction}) can be found in \cite{qu2017harnessing} Lemma 10.
	\end{proof}
	
	\begin{lemma}
		\label{lem: spectral radii and norms}
		(Adapted from Lemma 3 and Lemma 4 in \cite{pu2020push}) Suppose Assumption \ref{asp: stochastic} hold. 
		There exist vector norms $\|\cdot\|_{\mrR}$ and $\|\cdot\|_{\mrC}$, defined as $\|\mx\|_{\mrR}:=\|\tilde{\mR} \mx\|_2$ and $\|\mx\|_{\mrC}:=\|\tilde{\mC} \mx\|_2$ for all $\mx \in \mathbb{R}^n$, where $\tilde{\mR},\tilde{\mC}\in\mathbb{R}^{n\times n}$ are some reversible matrices, 
		such that $\tau_{\mrR}:=\|\mR_{\eta}-\frac{\mathbf{1}u^{\T}}{n}\|_{\mrR}<1$, $\tau_{\mrC}:=\|\mC_{\gamma}-\frac{v\mathbf{1}^{\T}}{n}\|_{\mrC}<1$,\footnote{With a slight abuse of notation,
			we do not distinguish between the vector norms on $\mathbb{R}^n$ and their induced
			matrix norms, e.g., for any matrix $\mM\in\mathbb{R}^{n\times n}$ and $\|\mM\|_{\mrR}:=\tilde{\mR}\mM\tilde{\mR}^{-1}$, $\|\mM\|_{\mrC}:=\tilde{\mC}\mM\tilde{\mC}^{-1}$.} and $\tau_{\mrR}$ and $\tau_{\mrC}$ are arbitrarily close to the spectral radii $\rho(\mR_{\eta}-\mathbf{1}u^{\T}/n)<1$ and $\rho(\mC_{\gamma}-v\mathbf{1}^{\T}/n)<1$, respectively. In addition, given any diagonal matrix $\mM\in\mathbb{R}^{n\times n}$, we have $\|\mM\|_{\mrR}=\|\mM\|_{\mrC}=\|\mM\|_2$.
	\end{lemma}
	
	The following two lemmas are also taken from \cite{pu2018push}.
	\begin{lemma}
		\label{lem: matrix norm production}
		Given an arbitrary norm $\|\cdot\|$, for any $\mW\in\mathbb{R}^{n\times n}$ and $\mx\in\mathbb{R}^{n\times p}$, we have $\|\mW\mx\|\le \|\mW\|\|\mx\|$. For any $w\in\mathbb{R}^{n\times 1}$ and $x\in\mathbb{R}^{1\times p}$, we have $\|wx\|=\|w\|\|x\|_2$.
	\end{lemma}
	\begin{lemma}
		\label{lem: norm equivalence}
		There exist constants $\delta_{\mrC,\mrR}, \delta_{\mrC,2}, \delta_{\mrR,\mrC}, \delta_{\mrR,2}>0$ such that for all $\mx\in\mathbb{R}^{n\times p}$, we have  $\|\mx\|_{\mrC}\le \delta_{\mrC,\mrR}\|\mx\|_{\mrR}$, $\|\mx\|_{\mrC}\le \delta_{\mrC,2}\|\mx\|_2$, $\|\mx\|_{\mrR}\le \delta_{\mrR,\mrC}\|\mx\|_{\mrC}$,  and $\|\mx\|_{\mrR}\le \delta_{\mrR,2}\|\mx\|_2$. In addition, without
		loss of generality, we can assume $\|\mx\|_2\le \|\mx\|_{\mrR}$ and $\|\mx\|_2\le \|\mx\|_{\mrC}$ for all $\mx$.
	\end{lemma}
	
	\subsection{Proof of Lemma \ref{lem: important inequalities}}
	\label{proof lem: important inequalities}
	
	The three inequalities embedded in (\ref{main ineqalities}) come from (\ref{ox pre}), (\ref{mx-ox pre}), and (\ref{my-oy pre}), respectively.
	
	\subsubsection*{First inequality}
	By Assumption \ref{asp: noises} and relation (\ref{opt_contraction}) in Lemma \ref{lem: Lipschitz implications},
	we obtain from (\ref{ox pre}) that
	{\small\begin{equation}
		\label{x_bar_error_pre}
		\begin{split}
		& \bE[\|\ox_{k+1}-x^*\|_2^2\mid \mathcal{F}_k]\\
		& \le \bE[\|\ox_k-\ta g_k-x^*-\ta (h_k-g_k)-\frac{\alpha}{n}u^{\T}(\my_k-v\oy_k)\|_2^2\mid \mathcal{F}_k]\\
		& \quad+\bE[\|\ta\frac{\gamma}{n}\mathbf{1}^{\T}\mepsilon_k\|_2^2\mid \mathcal{F}_k]+\bE[\|\frac{\eta}{n}u^{\T}\mxi_k\|_2^2\mid \mathcal{F}_k]\\
		& \le (1-\ta \mu)^2\|\ox_k-x^*\|_2^2+\|\ta (h_k-g_k)+\frac{\alpha}{n}u^{\T}(\my_k-v\oy_k)\|_2^2\\
		& \quad + 2(1-\ta \mu)\|\ox_k-x^*\|_2 \|\ta (h_k-g_k)+\frac{\alpha}{n}u^{\T}(\my_k-v\oy_k)\|_2\\
		& \quad +\frac{\ta^2\gamma^2\sigma_{\mepsilon}^2}{n^2}+\frac{\eta^2\|u\|_2^2 \sigma_{\mxi}^2}{n^2}.
		\end{split}
		\end{equation}}\normalsize
	Notice that
	\begin{multline*}
	2\|\ox_k-x^*\|_2 \|\ta (h_k-g_k)+\frac{\alpha}{n}u^{\T}(\my_k-v\oy_k)\|_2\\
	\le \ta\mu \|\ox_k-x^*\|_2^2 +\frac{1}{\ta \mu}\|\ta (h_k-g_k)+\frac{\alpha}{n}u^{\T}(\my_k-v\oy_k)\|_2^2,
	\end{multline*}
	and from (\ref{h_k-g_k}) in Lemma \ref{lem: Lipschitz implications},
	\begin{align*}
	& \|\ta (h_k-g_k)+\frac{\alpha}{n}u^{\T}(\my_k-v\oy_k)\|_2^2\\
	& \le 2\|\ta (h_k-g_k)\|_2^2+2\|\frac{\alpha}{n}u^{\T}(\my_k-v\oy_k)\|_2^2\\
	& \le \frac{2\ta^2 L^2\|\mx_k-\mathbf{1}\ox_k\|_2^2}{n}+\frac{2\alpha^2\|u\|_2^2 \|\my_k-v\oy_k\|_2^2}{n^2}.
	\end{align*}
	We have from (\ref{x_bar_error_pre}) that
	{\small\begin{align*}
		& \bE[\|\ox_{k+1}-x^*\|_2^2\mid \mathcal{F}_k]\\
		& \le (1-\ta \mu)\|\ox_k-x^*\|_2^2+\frac{1}{\ta \mu}\|\ta (h_k-g_k)+\frac{\alpha}{n}u^{\T}(\my_k-v\oy_k)\|_2^2\\
		& \quad + \frac{\ta^2\gamma^2\sigma_{\mepsilon}^2}{n^2}+\frac{\eta^2\|u\|_2^2 \sigma_{\mxi}^2}{n^2}\\
		& \le (1-\ta \mu)\|\ox_k-x^*\|_2^2+\frac{2\ta L^2\|\mx_k-\mathbf{1}\ox_k\|_2^2}{\mu n}\\
		& \quad +\frac{2\alpha^2\|u\|_2^2 \|\my_k-v\oy_k\|_2^2}{\ta \mu n^2}+ \frac{\ta^2\gamma^2\sigma_{\mepsilon}^2}{n^2}+\frac{\eta^2\|u\|_2^2 \sigma_{\mxi}^2}{n^2}\\
		& \le  (1-\ta \mu)\|\ox_k-x^*\|_2^2+c_1\alpha\|\mx_k-\mathbf{1}\ox_k\|_{\mrR}^2+c_2\alpha\|\my_k-v\oy_k\|_{\mrC}^2\\
		& \quad+c_3 \alpha^2\gamma^2\sigma_{\mepsilon}^2+c_4\eta^2\sigma_{\mxi}^2,
		\end{align*}}\normalsize
	where the last inequality is from Lemma \ref{lem: norm equivalence} and
	{\small\begin{align}
		\label{c1-c4}
		c_1:= \frac{2u^{\T}vL^2}{\mu n^2},\,c_2:=\frac{2\|u\|_2^2}{u^{\T}v\mu n},\,c_3:=\frac{(u^{\T}v)^2}{n^4},\,c_4:=\frac{\|u\|_2^2 }{n^2}.
		\end{align}}\normalsize
	Taking full expectation on both sides of the inequality completes the proof.
	
	\subsubsection*{Second inequality}
	By relation (\ref{mx-ox pre}) and Lemma \ref{lem: matrix norm production}, we see that
	{\small\begin{equation}
		\label{x_error_pre}
		\begin{split}
		& \bE[\|\mx_{k+1}-\mathbf{1}\ox_{k+1}\|_{\mrR}^2\mid \mathcal{F}_k]\\
		& \le \tau_{\mrR}^2\|\mx_k-\mathbf{1}\ox_k\|_{\mrR}^2+\alpha^2\|\mI-\frac{\mathbf{1}u^{\T}}{n}\|_{\mrR}^2\bE[\|\my_k\|_{\mrR}^2\mid \mathcal{F}_k]\\
		& \quad +2\alpha\tau_{\mrR}\|\mI-\frac{\mathbf{1}u^{\T}}{n}\|_{\mrR}\|\mx_k-\mathbf{1}\ox_k\|_{\mrR}\bE[\|\my_k\|_{\mrR}\mid \mathcal{F}_k]\\
		& \quad+\eta^2\bE[\|\left(\mI-\frac{\mathbf{1}u^{\T}}{n}\right)\mxi_k\|_{\mrR}^2]\\
		& \le \frac{(1+\tau_{\mrR}^2)}{2}\|\mx_k-\mathbf{1}\ox_k\|_{\mrR}^2\\
		& \quad +\frac{\alpha^2(1+\tau_{\mrR}^2)}{(1-\tau_{\mrR}^2)}\|\mI-\frac{\mathbf{1}u^{\T}}{n}\|_{\mrR}^2\bE[\|\my_k\|_{\mrR}^2\mid \mathcal{F}_k]\\
		& \quad +\eta^2\|\tilde{\mR}\left(\mI-\frac{\mathbf{1}u^{\T}}{n}\right)\|_2^2\bE[\|\mxi_k\|_2^2].
		\end{split}
		\end{equation}}\normalsize
	
	To bound $\bE[\|\my_k\|_{\mrR}^2\mid \mathcal{F}_k]$, note that
	{\small\begin{align*}
		\|\my_k\|_{\mrR}^2\le & 2\|\my_k-v\oy_k\|_{\mrR}^2+2\|v\oy_k\|_{\mrR}^2\\
		= & 2\|\my_k-v\oy_k\|_{\mrR}^2+2\|v\|_{\mrR}^2\|\oy_k\|_2^2,
		\end{align*}}\normalsize
	where the equality follows from Lemma \ref{lem: matrix norm production}.
	In light of Lemma \ref{lem: Lipschitz implications},
	{\small
		\begin{equation}
		\label{oy_k bound}
		\begin{split}
		& \bE[\|\oy_k\|_2^2\mid \mathcal{F}_k] = \bE[\|\oy_k-h_k\|_2^2]+\|h_k\|_2^2\\
		& \le \bE[\|\oy_k-h_k\|_2^2]+2\|h_k-g_k\|_2^2 + 2\|g_k\|_2^2\\
		& \le \frac{\gamma^2\sigma_{\mepsilon}^2}{n}+\frac{2L^2}{n}\|\mx_k-\mathbf{1}\ox_k\|_2^2+L^2\|\ox_k-x^*\|_2^2.
		\end{split}
		\end{equation}}\normalsize
	Hence
	{\small\begin{align*}
		& \bE[\|\my_k\|_{\mrR}^2\mid \mathcal{F}_k] \le 2\bE[\|\my_k-v\oy_k\|_{\mrR}^2\mid \mathcal{F}_k]\\
		& \quad +2\|v\|_{\mrR}^2\left(\frac{\gamma^2\sigma_{\mepsilon}^2}{n}+\frac{2L^2}{n}\|\mx_k-\mathbf{1}\ox_k\|_2^2+L^2\|\ox_k-x^*\|_2^2\right).
		\end{align*}}\normalsize
	
	Noticing that $\tau_{\mrR}<1$ and $\bE[\|\mxi_k\|_2^2]\le \sigma_{\mxi}^2$ from Assumption \ref{asp: noises}, relation (\ref{x_error_pre}) leads to
	{\small\begin{align*}
		& \bE[\|\mx_{k+1}-\mathbf{1}\ox_{k+1}\|_{\mrR}^2\mid \mathcal{F}_k]\le \frac{(1+\tau_{\mrR}^2)}{2}\|\mx_k-\mathbf{1}\ox_k\|_{\mrR}^2\\
		& \quad +\frac{4\alpha^2}{(1-\tau_{\mrR}^2)}\|\mI-\frac{\mathbf{1}u^{\T}}{n}\|_{\mrR}^2\bE[\|\my_k-v\oy_k\|_{\mrR}^2\mid \mathcal{F}_k]\\
		& \quad +\frac{4\alpha^2}{(1-\tau_{\mrR}^2)}\|\mI-\frac{\mathbf{1}u^{\T}}{n}\|_{\mrR}^2\|v\|_{\mrR}^2\\
		& \quad \cdot\left(\frac{\gamma^2\sigma_{\mepsilon}^2}{n}+\frac{2L^2}{n}\|\mx_k-\mathbf{1}\ox_k\|_2^2+L^2\|\ox_k-x^*\|_2^2\right)\\
		& \quad+\eta^2\|\tilde{\mR}\left(\mI-\frac{\mathbf{1}u^{\T}}{n}\right)\|_2^2\sigma_{\mxi}^2\\
		& \le  c_5\alpha^2 \|\ox_k-x^*\|_2^2 + \left[\frac{(1+\tau_{\mrR}^2)}{2}+c_6\alpha^2\right] \|\mx_k-\mathbf{1}\ox_k\|_{\mrR}^2\\ 
		& \quad + c_7\alpha^2\bE[\|\my_k-v\oy_k\|_{\mrC}^2\mid \mathcal{F}_k]+ c_8 \alpha^2\gamma^2\sigma_{\mepsilon}^2+c_9\eta^2\sigma_{\mxi}^2,
		\end{align*}}\normalsize
	where Lemma \ref{lem: norm equivalence} was invoked and constants are given by
	{\small\begin{align}
		\label{c5-c9}
		\begin{split}
		& c_5:=\frac{4 L^2}{(1-\tau_{\mrR}^2)}\|\mI-\frac{\mathbf{1}u^{\T}}{n}\|_{\mrR}^2\|v\|_{\mrR}^2\\
		& c_6:= \frac{8 L^2}{(1-\tau_{\mrR}^2)n}\|\mI-\frac{\mathbf{1}u^{\T}}{n}\|_{\mrR}^2\|v\|_{\mrR}^2\\
		& c_7:= \frac{4}{(1-\tau_{\mrR}^2)}\|\mI-\frac{\mathbf{1}u^{\T}}{n}\|_{\mrR}^2\delta_{R,C}\\
		& c_8:= \frac{4}{(1-\tau_{\mrR}^2)}\|\mI-\frac{\mathbf{1}u^{\T}}{n}\|_{\mrR}^2\|v\|_{\mrR}^2\frac{1}{n}\\
		& c_9:= \|\tilde{\mR}\left(\mI-\frac{\mathbf{1}u^{\T}}{n}\right)\|_2^2.
		\end{split}
		\end{align}}\normalsize
	Again, taking full expectation on both sides of the inequality completes the proof.
	
	\subsubsection*{Third inequality}
	It follows from (\ref{my-oy pre}), Lemma \ref{lem: matrix norm production} and Lemma \ref{lem: norm equivalence} that
	{\small
		\begin{equation}
		\label{y_error_pre}
		\begin{split}
		& \|\my_{k+1}-v\oy_{k+1}\|_{\mrC}^2 \le  \tau_{\mrC}^2\|\my_k-v\oy_k\|_{\mrC}^2\\ 
		& \quad +\|\mI-\frac{v\mathbf{1}^{\T}}{n}\|_{\mrC}^2\|\tilde{\nabla}F(\mx_{k+1})-\tilde{\nabla}F(\mx_k)\|_{\mrC}^2\\
		& \quad +2\tau_{\mrC}\|\mI-\frac{v\mathbf{1}^{\T}}{n}\|_{\mrC}\|\my_k-v\oy_k\|_{\mrC}\|\tilde{\nabla}F(\mx_{k+1})-\tilde{\nabla}F(\mx_k)\|_{\mrC}\\
		& \le \frac{(1+\tau_{\mrC}^2)}{2}\|\my_k-v\oy_k\|_{\mrC}^2\\
		& \quad +\frac{(1+\tau_{\mrC}^2)}{(1-\tau_{\mrC}^2)}\|\mI-\frac{v\mathbf{1}^{\T}}{n}\|_{\mrC}^2\delta_{C,2}^2\|\tilde{\nabla}F(\mx_{k+1})-\tilde{\nabla}F(\mx_k)\|_2^2.
		\end{split}
		\end{equation}
	}\normalsize
	We now bound $\bE[\tilde{\nabla}F(\mx_{k+1})-\tilde{\nabla}F(\mx_k)\|_2^2\mid \mathcal{F}_k]$.

		Denote $G_k:=\tilde{\nabla}F(\mx_k)$ and $\nabla_k:=\nabla F(\mx_k)$ for all $k$ for short. We have
		{\small\begin{equation}
			\label{G_k+1-G_k}
			\begin{split}
			& \bE[\|G_{k+1}-G_k\|_2^2\mid \mathcal{F}_k]\\
			& = \bE[\|\nabla_{k+1}-\nabla_k \|_2^2\mid \mathcal{F}_k]\\
			& \quad+2\bE[\langle \nabla_{k+1}-\nabla_k , G_{k+1}-\nabla_{k+1}-G_k+\nabla_k \rangle \mid \mathcal{F}_k]\\
			& \quad+\bE[\|G_{k+1}-\nabla_{k+1}-G_k+\nabla_k \|_2^2\mid \mathcal{F}_k]\\
			& \le \bE[\|\nabla_{k+1}-\nabla_k \|_2^2\mid \mathcal{F}_k]+2\bE[\langle \nabla_{k+1}, -G_k+\nabla_k \rangle \mid \mathcal{F}_k]\\
			& \quad +\sp{2\gamma^2\sigma_{\mepsilon}^2}.
			\end{split}
			\end{equation}}\normalsize
		where we invoked Assumption \ref{asp: noises} for the above inequality.
		
		From an argument that is similar to Lemma 8 in \cite{pu2020distributed}, we have $\bE[\langle \nabla_{k+1}, -G_k+\nabla_k \rangle \mid \mathcal{F}_k]\le \sp{\alpha \gamma^2 L\sigma_{\mepsilon}^2}$. In addition, from Assumption \ref{asp; strconvex Lipschitz} and Assumption \ref{asp: noises},
		{\small\begin{align*}
			& \bE[\|\nabla_{k+1}-\nabla_k \|_2^2\mid \mathcal{F}_k]\le L^2\bE[\|\mx_{k+1}-\mx_k\|_2^2\mid \mathcal{F}_k] \\
			& = L^2\bE[\|\mR_{\eta} \mx_k-\mx_k-\alpha\my_k+\eta\mxi_k\|_2^2\mid \mathcal{F}_k]\\
			& = L^2\bE[\|(\mR_{\eta}-\mI)(\mx_k-\mathbf{1}\ox_k)-\alpha(\my_k-v\oy_k)-\alpha v\oy_k\|_2^2\mid \mathcal{F}_k]\\
			& \quad + \eta^2 L^2\bE[\|\mxi_k\|_2^2\mid \mathcal{F}_k]\\
			& \le 3L^2\|\mR_{\eta}-\mI\|_2^2\|\mx_k-\mathbf{1}\ox_k\|_2^2 + 3\alpha^2 L^2\bE[\|\my_k-v\oy_k\|_2^2\mid \mathcal{F}_k]\\
			& \quad +3\alpha^2 L^2\|v\|_2^2\bE[\|\oy_k\|^2\mid \mathcal{F}_k]+\eta^2 L^2\sigma_{\mxi}^2,
			\end{align*}}\normalsize
			In light of relation (\ref{oy_k bound}), we further obtain
			{\small\begin{align*}
			& \bE[\|\nabla_{k+1}-\nabla_k \|_2^2\mid \mathcal{F}_k]\\
			& \le 3L^2\|\mR_{\eta}-\mI\|_2^2\|\mx_k-\mathbf{1}\ox_k\|_2^2 + 3\alpha^2 L^2\bE[\|\my_k-v\oy_k\|_2^2\mid \mathcal{F}_k]\\
			& \quad +3\alpha^2 L^2\|v\|_2^2\left(\frac{\gamma^2\sigma_{\mepsilon}^2}{n}+\frac{2L^2}{n}\|\mx_k-\mathbf{1}\ox_k\|_2^2+L^2\|\ox_k-x^*\|_2^2\right)\\
			& \quad +\eta^2 L^2\sigma_{\mxi}^2\\
			& = 3\alpha^2 L^2\bE[\|\my_k-v\oy_k\|_2^2\mid \mathcal{F}_k]\\
			& \quad + \left(3L^2\|\mR_{\eta}-\mI\|_2^2+\frac{6\alpha^2 L^4\|v\|_2^2}{n}\right)\|\mx_k-\mathbf{1}\ox_k\|_2^2\\
			& \quad + 3\alpha^2 L^4\|v\|_2^2\|\ox_k-x^*\|_2^2 + \frac{3\alpha^2\gamma^2 L^2 \|v\|_2^2\sigma_{\mepsilon}^2}{n}+\eta^2 L^2\sigma_{\mxi}^2.
			\end{align*}}\normalsize
		Hence relation (\ref{G_k+1-G_k}) leads to
		{\small\begin{align*}
			& \bE[\|G_{k+1}-G_k\|_2^2\mid \mathcal{F}_k] \le 3\alpha^2 L^2\bE[\|\my_k-v\oy_k\|_2^2\mid \mathcal{F}_k]\\
			& \quad + \left(3L^2\|\mR_{\eta}-\mI\|_2^2+\frac{6\alpha^2 L^4\|v\|_2^2 }{n}\right)\|\mx_k-\mathbf{1}\ox_k\|_2^2\\
			& \quad + 3\alpha^2 L^4\|v\|_2^2\|\ox_k-x^*\|_2^2 + \frac{3\alpha^2 \gamma^2 L^2\|v\|_2^2\sigma_{\mepsilon}^2}{n}+\eta^2 L^2\sigma_{\mxi}^2\\
			& \quad+\sp{2\gamma^2(\alpha L+1)\sigma_{\mepsilon}^2}.
			\end{align*}
		}

	In light of Lemma \ref{lem: norm equivalence}, the above inequality, and considering that $\alpha\le 1/L$, we have from (\ref{y_error_pre}) that
	{\small\begin{equation*}
	\begin{split}
	& \bE[\|\my_{k+1}-v\oy_{k+1}\|_{\mrC}^2\mid \mathcal{F}_k] \le c_{10}\alpha^2\|\ox_k-x^*\|_2^2\\
	& \quad  +c_{11}\|\mx_k-\mathbf{1}\ox_k\|_2^2+\left[\frac{(1+\tau_{\mrC}^2)}{2}+c_{12}\alpha^2\right]\bE[\|\my_k-v\oy_k\|_{\mrC}^2\mid \mathcal{F}_k]\\
	& \quad +c_{13}\gamma^2\sigma_{\mepsilon}^2+c_{14}\eta^2\sigma_{\mxi}^2,
	\end{split}
	\end{equation*}}\normalsize
    where
    {\small\begin{equation}
    	\label{c10-c13}
    	\begin{split}
    	& c_{10} := \frac{6}{(1-\tau_{\mrC}^2)}\|\mI-\frac{v\mathbf{1}^{\T}}{n}\|_{\mrC}^2\delta_{C,2}^2\|v\|_2^2 L^4\\
    	& c_{11}:= \frac{2}{(1-\tau_{\mrC}^2)}\|\mI-\frac{v\mathbf{1}^{\T}}{n}\|_{\mrC}^2\delta_{C,2}^2\left(3L^2\|\mR_{\eta}-\mI\|_2^2+\frac{6L^2\|v\|_2^2}{n}\right)\\
    	& c_{12} :=\frac{6}{(1-\tau_{\mrC}^2)}\|\mI-\frac{v\mathbf{1}^{\T}}{n}\|_{\mrC}^2\delta_{C,2}^2 L^2\\
    	& c_{13} := \frac{2}{(1-\tau_{\mrC}^2)}\|\mI-\frac{v\mathbf{1}^{\T}}{n}\|_{\mrC}^2\delta_{C,2}^2\left(\frac{3\|v\|_2^2}{n} +\sp{4}\right)\\
    	& c_{14} := \frac{2}{(1-\tau_{\mrC}^2)}\|\mI-\frac{v\mathbf{1}^{\T}}{n}\|_{\mrC}^2\delta_{C,2}^2 L^2. 
    	\end{split}
    	\end{equation}}\normalsize
    Taking full expectation yields the desired result.
	
\end{document}